\documentclass[11pt,leqno]{amsart}
\usepackage{amssymb,amsmath,amsthm,dsfont,fourier}
\usepackage{stmaryrd}

\usepackage{mathtools}
\usepackage{array}
\usepackage[text={6.5in,9in},centering]{geometry}
\usepackage[dvipsnames]{xcolor}
\usepackage{enumitem}
\usepackage{tikz-cd}
\usetikzlibrary{decorations.pathmorphing}
\usepackage[colorlinks=true,citecolor=NavyBlue,linkcolor=NavyBlue,urlcolor=NavyBlue]{hyperref}

\newcommand{\aar}[2][]{\xrightarrow[#1]{#2}}

\newcommand{\kker}{\operatorname{ker}}

\newtheorem{theorem}{Theorem}[section]
\newtheorem{proposition}[theorem]{Proposition}

\newtheorem{corollary}[theorem]{Corollary}
\theoremstyle{definition}
\newtheorem{definition}{Definition}[section]
\newtheorem{remark}[theorem]{Remark}
\newtheorem{example}[theorem]{Example}

\numberwithin{equation}{section}
\DeclareMathOperator{\Spec}{Spec}
\DeclareMathOperator{\Ass}{Ass}
\DeclareMathOperator{\MaxAss}{MaxAss}
\DeclareMathOperator{\length}{length}
\newcommand{\cP}{\mathcal P}

\begin{document}
\let\WriteBookmarks\relax
\def\floatpagepagefraction{1}
\def\textpagefraction{.001}

\title{A Mayer-Vietoris calculus for regular denominators}

\author{Elena Caviglia}
\address{[1] Department of Mathematics, Stellenbosch University, South Africa; [2] National Institute for Theoretical and Computational Sciences (NITheCS), South Africa}
\email{elena.caviglia@outlook.com}

\author{Amartya Goswami}
\address{[1] Department of Mathematics and Applied Mathematics, University of Johannesburg,
P.O. Box 524, Auckland Park 2006, South Africa; [2] National Institute for Theoretical and Computational Sciences (NITheCS), South Africa}
\email{agoswami@uj.ac.za}

\begin{abstract}
	Let $A$ be a commutative ring and let $I$ be an ideal. An element $a\in A$ is a regular denominator for $I$ when multiplication by $a$ on $A/I$ is injective; we denote the set of all such elements by $S_I$. We study how the common regular denominators for two ideals $I$ and $J$ are related to $I\cap J$ and $I+J$. This yields a particularly simple description when $I$ and $J$ are comaximal. Over Noetherian rings, the same viewpoint classifies denominator-equivalence classes by finite nonempty antichains of prime ideals, gives bounds for the associated primes of an intersection and leads to a local length identity. Furthermore, we show that flat base change preserves regular denominators, faithful flatness reflects them, and finite locally free quotients have fiberwise regular loci defined by determinants satisfying a multiplicative Mayer--Vietoris formula.  Our results provide alternatives to primary-decomposition computations in many concrete situations.
\end{abstract}


\maketitle

\section{Introduction}
When working modulo an ideal, not every element can be used safely as a denominator. Given a commutative ring $A$ and an ideal $I\subseteq A$, we say that $a\in A$ is a \emph{regular denominator} for $I$ if multiplication by $a$ on $A/I$ is injective; equivalently, the image of $a$ is a non-zero-divisor in $A/I$. We denote the set of all regular denominators for $I$ by
\[
S_I
:=\{a\in A\mid \mu_a^I\colon A/I\to A/I
\text{ is injective}\}
=\{a\in A\mid(I:a)=I\}.
\]
Notice that the set $S_I$ is multiplicatively closed, and the images of its elements in $A/I$ are precisely the regular elements inverted in the total quotient ring of $A/I$. Therefore, regular denominators belong to the basic framework of localization, primary decomposition, and the study of zero divisors. Their fundamental properties are developed in standard texts on commutative algebra; see, for example, \cite{AM69}, \cite{Eis95}, and \cite{Mat86}. Much of the theory about regular denominators considers one ideal at a time. The purpose of the present paper is to understand what happens when the regular denominators with respect to two ideals  of the same ring are considered simultaneously. In particular, given two ideals $I$ and $J$ of a commutative ring $A$, we want to study their common regular denominator by analizing those of the symmetric pair $I\cap J$ and $I+J$, without studying $I$ and $J$ separately.

Our starting point is the familiar short exact sequence
$$
0\longrightarrow A/(I\cap J)\longrightarrow A/I\oplus A/J
\longrightarrow A/(I+J)\longrightarrow 0
$$
that is the algebraic form of a Mayer--Vietoris construction. Fix an element $a\in A$. Multiplication by $a$ defines compatible endomorphisms of all three terms, so the snake lemma produces a six-term exact sequence relating the kernels and cokernels of multiplication on the four quotients determined by $I$, $J$, $I\cap J$, and $I+J$.
The resulting sequence carries concrete information about regular denominators. In particular, the kernel of its connecting homomorphism is described by the colon ideals $(I:a)$ and $(J:a)$. Consequently, once $a$ is known to be regular modulo $I\cap J$, the injectivity of the connecting map is equivalent to $a$ being regular modulo both $I$ and $J$. When $I$ and $J$ are comaximal, the comparison simplifies further: determining whether $a$ is a common regular denominator for $I$ and $J$ reduces to checking regularity only on $A/(I\cap J)$.

In the Noetherian case, regular denominators can be read directly from associated primes: the zero divisors of a finite module are precisely the elements lying in the union of its associated primes \cite[Theorem~6.1(ii), p.~38]{Mat86}. This turns the Mayer--Vietoris calculus developed in Section~\ref{sec:mv} into a study of finite collections of prime ideals.

More precisely, over a Noetherian ring, two proper ideals have the
same regular-denominator set if and only if their maximal associated
primes determine the same finite antichain in $\Spec(A)$. It follows that the denominator-equivalence classes of proper ideals are in bijection with the finite nonempty antichains in $\Spec(A)$ (see Proposition \ref{propdeneqantich}). Furthermore, each such class has a canonical radical representative. We will see that this radical representative does not coincide with the radical of the ideal in general. Indeed, it is determined by maximal associated primes and thus it retains information coming from embedded primes.

The same Mayer--Vietoris type short exact sequence also controls the associated primes of an intersection. Indeed, we prove (Proposition \ref{propass}) the following inclusions 
$$
\bigl(\Ass_A(A/I)\cup\Ass_A(A/J)\bigr)
\setminus\Ass_A(A/(I+J))
\subseteq
\Ass_A(A/(I\cap J))
\subseteq
\Ass_A(A/I)\cup\Ass_A(A/J).
$$
Under additional hypotheses guarantee that the relevant modules have finite length of the involved modules, the Mayer-Vietoris sequence also gives an alternating-sum formula involving the four kernels and four cokernels of multiplication (Proposition \ref{propformula}). Taken together, these results recover embedded-prime information that is not available from the topological support alone.

In Section~\ref{sec:basechange}, we study the behavior of regular denominators under base change. For this, we consider an $A$-algebra $B$ that represents a family of spaces living over the base space $\Spec(A)$. We show that flat base change preserves regularity, while faithful flatness allows it to be detected after passing to the larger ring. 
When $B/I$ is finite locally free over the base ring $A$, multiplication by an element $b\in B$ has a well-defined determinant in $A$. We show that this determinant records exactly where $b$ remains regular in the fibers: the relevant primes are precisely those of the principal open set associated to the determinant. This result fits into the theory of determinant and norm constructions for finite locally free modules; see, for example, \cite{KM76,Fer98}.

Using again the Mayer-Vietoris sequence, we relate these determinants for the four ideals $I$, $J$, $I\cap J$, and $I+J$, obtaining a multiplicative formula for the determinants and hence for the corresponding fiberwise regular loci (Theorem \ref{detthm}).

Our results also have a computational motivation. Determining associated primes and finding primary decompositions are standard problems in computational commutative algebra, with classical algorithms developed in \cite{GTZ88,EHV92}; at the same time, polynomial-ideal computations can become very complex in general \cite{MM82}. Although the results of our paper do not change that general complexity picture, they do provide, in situations where the hypotheses apply, a way to avoid some of the heavier computations. Questions about regular denominators can then be reduced to colon ideals, sums, intersections, equality tests, or determinants---operations that are already part of the standard Gr\"obner-basis and linear-algebra toolkit; see, for example, \cite{GP08}. The final section illustrates this point through explicit examples, showing in each case exactly which larger computation can be replaced by a simpler one.

The paper is organized as follows. In Section~\ref{sec:mv},  we introduce the Mayer-Vietoris exact sequence and derive its first consequences for common regular denominators. The purpose of Section~\ref{sec:noeth} is to consider the Noetherian case, where we study denominator-equivalence classes, associated primes, and the resulting length formula. In Section~\ref{sec:basechange}, we examine the behavior of regular denominators under flat base change and describe the corresponding determinant loci. Finally, in Section~\ref{sec:alg}, we explain how these results can be used computationally and illustrate the resulting simplifications through examples.

\section{The Mayer--Vietoris sequence for regular denominators}\label{sec:mv}
Let $A$ be a commutative ring with unit and let $I$ be an ideal of $A$. Given an element $a\in A$, we consider multiplication by $a$ on $A/I$. This is the $A$-module homomorphism 
$$\mu_a^I\colon A/I \rightarrow A/I$$ 
that sends $x+I\in A/I$ to $ax+I\in A/I$. The kernel and the cokernel of $\mu_a^I$ are respectively 
\begin{align*}
K_a^I:&= \kker(\mu^I_a)=(I:a)/I	\\
\noalign{\hbox{and}}
C_a^I&:={\operatorname{coker}(\mu_a^I)=\frac{A/I}{a(A/I)}\cong A/(I+aA)}.
\end{align*}
We write
\[
S_I:=\{a\in A\mid K_a^I=0\}=\{a\in A\mid (I:a)=I\}
\]
for the regular-denominator set of $I$. Thus $a\in S_I$ exactly when its class is a non-zero-divisor on $A/I$.

The following theorem applies the snake lemma to produce the Mayer--Vietoris exact sequence for regular denominators. This result will play a crucial role in the rest of the paper. 

\begin{theorem}\label{thmMayerVietsequence}
    Let $I,J$ be ideals of $A$ and let $a\in A$. There is an exact sequence 
    $$0\,\aar{} K^{I \cap J}_a \aar{\eta_a} {K_a^I \oplus K_a^J} \aar{\alpha_a} K_a^{I+J}  \aar{\delta_a} C_a^{I \cap J} \aar{\beta_a} {C_a^I \oplus C_a^J} \aar{\gamma_a} C_a^{I+J} \aar{} 0 $$
    where 
\begin{enumerate}
\item[$\bullet$] $\eta_a\bigl(c+(I\cap J)\bigr)
    = (c+I,c+J),$
\item[$\bullet$] $\alpha_a(d+I,b+J)
  =d-b+(I+J),$
\item[$\bullet$]    $\beta_a\bigl(d+(I\cap J+aA)\bigr)
   =(d+(I+aA),d+(J+aA)),$
\item[$\bullet$] $\gamma_a(d+(I+aA),b+(J+aA))
    =d-b+(I+J+aA),$

    \item[$\bullet$] $\delta_a\bigl(c+(I+J)\bigr)=j+(I\cap J+aA)
    $ with $i\in I$ and $j\in J$ such that $ac=i+j$.
\end{enumerate}
\end{theorem}

\begin{proof}
    Consider the well-known short exact sequence (see, for instance, \cite[Lemma 29.4.6, Tag 0C4J]{Stk26})
    $$0\, \aar{} A/(I\cap J) \aar{\delta} {A/I \oplus A/J} \aar{\sigma} A/(I+J)\, \aar{} 0$$
    with $\delta(x+(I\cap J))=(x+I, x+J)$ and $\sigma (y+I,z+J)=((y-z)+(I+J))$. Following the notation introduced above, the following is a morphism of short exact sequences
   \[\begin{tikzcd}[ampersand replacement=\&]
	0 \& {A/(I\cap J)} \& {A/I \oplus A/J} \& {A/(I+J)} \& 0 \\
	0 \& {A/(I\cap J)} \& {A/I \oplus A/J} \& {A/(I+J)} \& 0
	\arrow[from=1-1, to=1-2]
	\arrow["\delta", from=1-2, to=1-3]
	\arrow["{\mu_a^{I \cap J}}", from=1-2, to=2-2]
	\arrow["\sigma", from=1-3, to=1-4]
	\arrow["{\mu_a^I\oplus \mu_a^J}", from=1-3, to=2-3]
	\arrow[from=1-4, to=1-5]
	\arrow["{\mu_a^{I+J}}", from=1-4, to=2-4]
	\arrow[from=2-1, to=2-2]
	\arrow["\delta", from=2-2, to=2-3]
	\arrow["\sigma", from=2-3, to=2-4]
	\arrow[from=2-4, to=2-5]
\end{tikzcd}\]
since the two squares trivially commute by definition. The snake lemma then yields the desired exact sequence
$$0\,\aar{} K^{I \cap J}_a \aar{\eta_a} {K_a^I \oplus K_a^J} \aar{\alpha_a} K_a^{I+J}  \aar{\delta_a} C_a^{I \cap J} \aar{\beta_a} {C_a^I \oplus C_a^J} \aar{\gamma_a} C_a^{I+J} \aar{} 0.$$
Notice that the connecting homomorphism $\delta_a$ is precisely the map obtained by the corresponding diagram chase when applying the snake lemma.
\end{proof}

\begin{remark}\label{remkerdelta}
    By construction, the kernel of the connecting homomorphism $\delta_{a}$ of Theorem \ref{thmMayerVietsequence} is given by 
    $$\ker(\delta_a)=((I:a)+(J:a))/(I+J).$$
    This means that the injectivity of $\delta_a$ is determined by whether $(I:a)+(J:a)=I+J$. This fact will be useful later.
\end{remark}

 We can now use the long exact sequence of Theorem \ref{thmMayerVietsequence} to develop effective tools to compute $S_I \cap S_J$. The following results are very useful to this end.

 \begin{proposition}\label{propeqfacts}
     Let $I,J$ be ideals of $A$ and let $a\in S_{I\cap J}$. The following facts are equivalent:
     \begin{enumerate}
         \item $a\in S_I \cap S_J$;
         \item $(I:a)+(J:a)=I+J$.
     \end{enumerate}
 \end{proposition}
\begin{proof}
	Since $a\in S_{I\cap J}$, we have
	$
	K_a^{I\cap J}=0.
	$
	Hence the exact sequence of Theorem~\ref{thmMayerVietsequence} begins as
	\[
	0\longrightarrow 0
	\longrightarrow K_a^I\oplus K_a^J
	\xrightarrow{\alpha_a} K_a^{I+J}
	\xrightarrow{\delta_a} C_a^{I\cap J}.
	\]
	Exactness at $K_a^I\oplus K_a^J$ shows that $\alpha_a$ is injective, while
	exactness at $K_a^{I+J}$ gives
	$
	\operatorname{Im}(\alpha_a)=\ker(\delta_a).
	$
	Consequently, restriction of the codomain of $\alpha_a$ yields an
	isomorphism
	\[
	K_a^I\oplus K_a^J \cong \ker(\delta_a).
	\]
	
	Now
	$
	a\in S_I\cap S_J
	$
	if and only if
	$
	K_a^I=K_a^J=0,
	$
	which is equivalent to
	$
	K_a^I\oplus K_a^J=0.
	$
	By the preceding isomorphism, this is equivalent to
	$
	\ker(\delta_a)=0.
	$
	By Remark~\ref{remkerdelta},
	\[
	\ker(\delta_a)
	=
	\frac{(I:a)+(J:a)}{I+J}.
	\]
	Therefore,
	$
	\ker(\delta_a)=0
	$
	if and only if
	\[
	(I:a)+(J:a)=I+J.
	\]
	This proves the equivalence.
\end{proof}

The next monomial calculation shows how the colon-sum criterion detects the failure of common regularity without requiring associated-prime computations for either quotient.

\begin{example}\label{monomialex}
Suppose $A=k[u,x,y]$. Let
\(
I=(ux,xy)\),  $J=(uy,xy)$, and $a=u$.
These are monomial ideals, and
$
I\cap J=(xy)$ and $(I\cap J:u)=(xy)$,
so $u\in S_{I\cap J}$. On the other hand,
$
(I:u)=(x)$, $(J:u)=(y)$, and
$(I:u)+(J:u)=(x,y)$,
whereas $I+J=(ux,uy,xy)$. Thanks to Proposition \ref{propeqfacts}, we can detect that $u\notin S_I\cap S_J$ by a simple comparison of explicitly generated ideals.
\end{example}

\begin{proposition}\label{propchain}
    Let $I,J$ be ideals of $A$. The following chain of inclusions holds:
    $$S_{I\cap J}\cap S_{I+J}
\subseteq S_I\cap S_J
\subseteq S_{I\cap J}.$$
In particular, if $I+J=A$ then 
$$S_{I} \cap S_J=S_{I\cap J}.$$
\end{proposition}
\begin{proof}
    Let $a\in S_{I\cap J} \cap S_{I+J}$. Then $K_a^{I\cap J}=0$ and $K_a^{I+J}=0$. Hence the long exact sequence of Theorem \ref{thmMayerVietsequence} yields 
    $$0\, \aar{} 0 \aar{\eta_a} K_a^I \oplus K_a^J \aar{\alpha_a} 0$$
    which implies $K_a^I=K_a^J=0$ and hence $a\in S_I\cap S_J$. 

    Consider now $a\in S_I\cap S_J$. We then have $K_a^I\oplus K_a^J=0$ and so the long exact sequence of Theorem \ref{thmMayerVietsequence} yields $K_a^{I\cap J}=0$, from which we conclude $a\in S_{I \cap J}$.

    Finally, if $I+J=A$ then $S_{I+J}=A$ and so the equality $S_{I} \cap S_J=S_{I\cap J}$ follows immediately from the chain of inclusions. 
\end{proof}

The inclusions in Proposition~\ref{propchain} can both be strict. The following examples show what causes the failure of equality in each case.

\begin{example}
Take $A=k[x,y]$, $I=(x)$, $J=(y)$, and $a=x+y$. Since $I\cap J=(xy)$, the class of $a$ is regular modulo $I\cap J$: it belongs to neither associated prime $(x)$ nor $(y)$. Its classes modulo $I$ and $J$ are respectively $y$ and $x$, so $a\in S_I\cap S_J$. But $I+J=(x,y)$ and the class of $a$ in $A/(I+J)=k$ is zero, hence $a\notin S_{I+J}$. This makes the first inclusion of Proposition~\ref{propchain} strict.
\end{example}

\begin{example}
Consider the same ideals $I$ and $J$ from Example \ref{monomialex}. The polynomial $u\in k[u,x,y]$ is in $S_{I\cap J}\setminus(S_I\cap S_J)$, so the second inclusion of Proposition~\ref{propchain} can be strict.
\end{example}

\begin{remark}
    Proposition \ref{propchain} shows that the common regular denominators of $I$ and $J$ can be studied by examining the regular denominators of $I\cap J$ and $I+J$. In particular, if $I$ and $J$ are comaximal, their common regular denominators are precisely the regular denominators of the intersection $I\cap J$. This result is helpful in explicitly computing the common regular denominators both by hand and using a computer. 
\end{remark}

\section{The Noetherian case}\label{sec:noeth}
In this section we specialize to the Noetherian case. In this context, the purely algebraic Mayer-Vietoris calculus developed in Section~\ref{sec:mv} translates into the geometric framework of affine schemes. In this geometric context, tracking the regular denominators of an ideal is equivalent to tracking the functions that avoid the irreducible components and embedded subschemes of the space.

The following standard characterization (see, for instance, \cite[Theorem 6.1(ii), p. 38]{Mat86}) of the zero divisors on a finite module over a Noetherian ring as the union of its associated prime ideals will be useful for us.

\begin{theorem}\label{thmcharzerodiv}
Let $A$ be a Noetherian ring and let $M$ be a finite $A$-module. The set of zero divisors for $M$ is the union of its associated prime ideals.
\end{theorem}

The following quotient exhibits the role of an embedded associated prime in determining the complete zero-divisor set.

\begin{example}
Let $A=k[x,y]$ and $M=A/(x^2,xy)$. The decomposition
$
(x^2,xy)=(x)\cap(x^2,y)
$
shows that
$
\Ass_A(M)=\{(x),(x,y)\}.
$
The embedded prime $(x,y)$ can also be seen directly: the annihilator of the nonzero class of $x$ in $M$ is $(x,y)$. The theorem therefore gives the zero-divisor set
$
(x)\cup(x,y)=(x,y).
$
Thus an element such as $1+y$ is regular on $M$, whereas every element of the maximal ideal $(x,y)$ is a zero divisor.
\end{example}

The following characterization of regular denominator sets for Noetherian rings follows immediately from Theorem \ref{thmcharzerodiv}.
\begin{corollary}\label{corollcharS_INoeth}
    Let $I$ be a proper ideal in a Noetherian ring $A$. The following equalities hold:
    $$S_I = A \setminus \bigcup_{P \in \Ass_A(A/I)} P=A \setminus \bigcup_{P \in \MaxAss_A(A/I)}P.$$
\end{corollary}

\begin{proof}
    The first equality is given by Theorem \ref{thmcharzerodiv} applied to the $A$-module $A/I$. The second equality follows from the fact that the union of the associated primes is clearly equal to the union of the maximal associated primes.
\end{proof}

The next example shows how Corollary~\ref{corollcharS_INoeth} reduces the question of regularity to checking membership in finitely many prime ideals.

\begin{example}
For $A=k[x,y]$ and $I=(xy)$, we have
\[
\Ass_A(A/I)=\MaxAss_A(A/I)=\{(x),(y)\}.
\]
Corollary~\ref{corollcharS_INoeth} gives
\[
S_I=A\setminus\bigl((x)\cup(y)\bigr).
\]
For example, $x+y\in S_I$, while $x$, $y$, and $xy$ are excluded. In general, to check whether an element belongs to $S_I$, it suffices to check whether it is a multiple of $x$ or $y$.  
\end{example}

We now define an equivalence relation on the proper ideals of the ring $A$ that identifies the ideals with the same regular denominators.

\begin{definition}
    Let $I$ and $J$ be ideals of $A$. We say that $I$ is \emph{denominator equivalent} to $J$, and we write $I \sim J$, if $S_I=S_J$.  
\end{definition}

The relation $\sim$ is clearly an equivalence relation on the set of proper ideals of $A$. We will prove that, when $A$ is Noetherian, the denominator-equivalence classes of proper ideals are in bijection with the finite nonempty antichains of prime ideals of $A$.

\begin{proposition}
\label{propdeneqantich}
Let $A$ be a Noetherian ring. 
The assignment 
$$[I]\longmapsto\MaxAss_A(A/I)$$
gives a bijection between the denominator-equivalence classes of proper ideals of $A$ and the finite non-empty antichains in $\Spec(A)$. 
\end{proposition}

\begin{proof}
Thanks to Corollary \ref{corollcharS_INoeth}, we have $$S_I = A \setminus \bigcup_{P \in \MaxAss_A(A/I)}P.$$
Moreover, since $A$ is Noetherian, the set $\Ass_A(A/I)$ is finite (see \cite[Theorem 6.5(i), p. 39]{Mat86}) and hence also the set $\MaxAss_A(A/I)$ is finite.
Notice also that the set $\MaxAss_A(A/I)$ is an antichain in $\Spec(A)$ by definition.

Assume now $I \sim J$. Then clearly  $$\bigcup_{P \in \MaxAss_A(A/I)}P= \bigcup_{P \in \MaxAss_A(A/J)}P.$$
By the Prime Avoidance Lemma, since the sets $\MaxAss_A(A/I)$ and $\MaxAss_A(A/J)$ are antichains, we conclude that $\MaxAss_A(A/I)=\MaxAss_A(A/J)$ and hence the assignment is well-defined. To show that it gives a bijection, we construct its inverse. Given a finite non-empty antichain $\cP\subseteq\Spec(A)$, consider the ideal $I_{\cP}=\bigcap_{P\in\cP}P$. 
Since it is an intersection of prime ideals, $I_{\cP}$ is radical. Furthermore, the minimal primes over $I_{\cP}$ are exactly the elements of $\cP$. Indeed, any prime containing $I_{\cP}$ must contain one of the primes in $\cP$ and two primes in $\cP$ are never contained in each other because $\cP$ is an antichain. On the other hand, since $A$ is Noetherian, the minimal primes of the radical ideal $I_{\cP}$ are exactly its maximal associated primes (see \cite[Theorem 6.5(iii), p. 39]{Mat86}).
Hence $$\MaxAss_A(A/I_{\cP}) = \cP.$$ This proves that $I_{\cP}$ is the unique radical ideal in its denominator-equivalence class. We therefore define the inverse of the starting assignment as the function sending $\cP$ to the denominator equivalence class $[I_{\cP}]$ of the ideal $I_{\cP}$. The equality $\MaxAss_A(A/I_{\cP}) = \cP$ then shows  that both composites of the two assignments are identities.
\end{proof}

The next example shows that denominator-equivalent ideals can define quite different schemes. In particular, a nonreduced quotient may have the same regular denominators as a much simpler reduced one.

\begin{example}
In $A=k[x,y]$, the ideals
$
I=(xy)$ and $J=(x^2y^3)=(x^2)\cap(y^3)
$
have the same maximal associated-prime antichain $\{(x),(y)\}$. Proposition~\ref{propdeneqantich} therefore gives $I\sim J$, although $A/I$ is reduced and $A/J$ has nontrivial nilpotents. Their common denominator class is represented by the radical ideal
$
(x)\cap(y)=(xy).
$
Thus a denominator calculation for the nonreduced quotient $A/J$ may be replaced by the squarefree monomial quotient $A/(xy)$.
\end{example}

\begin{corollary}\label{corollradical}
    Let $I$ be a proper ideal in a Noetherian ring $A$. Then 
$$S_I=S_{\bigcap_{P\in\MaxAss_A(A/I)}P}.$$
\end{corollary}

\begin{proof}
    It immediately follows from the proof of Proposition \ref{propdeneqantich}.
\end{proof}

The next example shows how the canonical radical representative from Corollary~\ref{corollradical} can differ from the ordinary radical, with the difference coming from an embedded associated prime.

\begin{example}
Consider $A=k[x,y]$ and $I=(x^2,xy)$. Here
\[
\Ass_A(A/I)=\{(x),(x,y)\}\quad \text{and}\quad
\MaxAss_A(A/I)=\{(x,y)\}.
\]
Consequently, Corollary~\ref{corollradical} yields
$
S_I=S_{(x,y)},
$
whereas $\sqrt I=(x)$. The regular-denominator computation is therefore reduced to the maximal-ideal quotient $A/(x,y)=k$, and this example also displays the effect of the embedded prime.
\end{example}

\begin{remark}\label{rem:den-radical}
Notice that the unique radical representative $I_{\mathcal{P}}$ need not coincide with the  usual radical $\sqrt{I}$. The classical radical $\sqrt{I}$  is the intersection of the minimal primes of $I$. In contrast, the ideal $I_{\mathcal{P}}$ is determined by the maximal associated primes. Because every minimal prime is contained in a maximal associated prime, we always have a structural inclusion $\sqrt{I} \subseteq I_{\mathcal{P}}$, with equality holding if and only if $A/I$ has no embedded primes. This algebraically formalizes the fact that regular denominators are deeply sensitive to embedded components, while the topological support is not.
\end{remark}

The following result provides a bound for the associated primes of the intersection of two ideals. Interestingly, the result can be derived using the same standard short exact sequence that we used in the proof of Theorem \ref{thmMayerVietsequence}.

\begin{proposition}
\label{propass}
Let $A$ be a Noetherian ring and let $I$ and $J$ be ideals of $A$. The following chain of inclusions holds:
\begin{align*}
&\bigl(\Ass_A(A/I)\cup\Ass_A(A/J)\bigr)
\setminus\Ass_A(A/(I+J))
\subseteq\Ass_A(A/(I\cap J))
\subseteq\Ass_A(A/I)\cup\Ass_A(A/J).
\end{align*}
\end{proposition}

\begin{proof}
    Consider once again the short exact sequence 
    $$0\, \aar{} A/(I\cap J) \aar{\delta} {A/I \oplus A/J} \aar{\sigma} A/(I+J)\, \aar{} 0$$
    with $\delta(x+(I\cap J))=(x+I, x+J)$ and $\sigma (y+I,z+J)=(y-z)+(I+J)$. 
Thanks to a standard result in commutative algebra (\cite[Theorem~6.3, p.~38]{Mat86}), given a short exact sequence $0 \to M' \to M \to M'' \to 0$, there are inclusions $$\Ass_A(M') \subseteq \Ass_A(M) \hspace{2ex} \text{ and } \hspace{2ex}\Ass_A(M) \subseteq \Ass_A(M') \cup \Ass_A(M'').$$ Applying the first inclusion to the previous short exact sequence, we obtain $$\Ass_A \left( A/(I\cap J)\right) \subseteq \Ass_A\left(A/I \oplus A/J\right).$$
Since the associated primes of a direct sum are the union of the individual associated primes (see also \cite[Theorem~6.3, p.~38]{Mat86}), we obtain the inclusion

$$\Ass_A(A/(I\cap J))
\subseteq\Ass_A(A/I)\cup\Ass_A(A/J).$$
Applying the second inclusion to our short exact sequence gives
$$\left( \Ass_A\left(A/I\right) \cup \Ass_A\left(A/J\right) \right) \setminus \Ass_A\left(A/(I+J)\right) \subseteq \Ass_A\left(A/(I \cap J)\right)$$
which is the first inclusion of the chain.
\end{proof}

The next example shows that both bounds in Proposition~\ref{propass} can be attained simultaneously, even for two simple coordinate ideals.

\begin{example}
Let $A=k[x,y]$, $I=(x)$, and $J=(y)$. Then $\Ass_A(A/I)=\{(x)\}$ and $\Ass_A(A/J)=\{(y)\}$. Moreover,
$\Ass_A(A/(I+J))=\{(x,y)\}$ and $\Ass_A(A/(I\cap J))=\{(x),(y)\}$. Hence both inclusions in Proposition~\ref{propass} are equalities. In this case, the associated primes of the intersection are recovered from the two prime quotients and the sum, without decomposing $(xy)$.
\end{example}

We now restrict our focus to a Noetherian local ring $(A, \mathcal{M})$ and assume the ideal $I \cap J + aA$ is $\mathcal{M}$-primary. Under this condition, the modules in the exact sequence of Theorem \ref{thmMayerVietsequence} become Artinian. This allows us to derive the following alternating sum formula.

\begin{proposition}
\label{propformula}
Let $(A,\mathcal{M})$ be a Noetherian local ring, let $I$ and $J$ be ideals of $A$ and let $a\in A$. If $(I\cap J)+aA$ is $\mathcal{M}$-primary, the following formula holds:
\begin{align*}
\length_A (K^{I \cap J}_a)
-\length_A (K^{I}_a)
-\length_A (K^{J}_a)
+\length_A (K^{I+J}_a)\\
&\hspace{-10em}=
\length_A (C^{I \cap J}_a)
-\length_A (C^{I}_a)
-\length_A (C^{J}_a)
+\length_A (C^{I+J}_a).	
\end{align*}
\end{proposition}

\begin{proof}
We first prove that all the modules involved in the formula have finite length. Since $(I \cap J) + aA$ is $\mathcal{M}$-primary, the quotient ring $A / ((I \cap J) + aA)$ is an Artinian ring. Thus, $C_a^{I \cap J}$ has finite length as an $A$-module. Moreover, the modules $C_a^I$, $C_a^J$ and $C_a^{I+J}$ are quotients of $C_a^{I \cap J}$ and so they also have finite length. On the other hand, since $(I \cap J) + aA$ is $\mathcal{M}$-primary, there exists an integer $n > 0$ such that $\mathcal{M}^n \subseteq (I \cap J) + aA$. As a consequence, $\mathcal{M}^n$ annihilates the module $K_a^{I \cap J}$. Hence, the module $K_a^{I \cap J}$ has finite length. Analogously, we conclude that also $K_a^I$, $K_a^J$ and $K_a^{I+J}$ have finite length since $I\cap J$ is contained in $I$, in $J$ and in $I+J$.

The alternating sum formula is then given by the additivity of length on finite exact sequences of modules (see \cite[Proposition~6.9, p.~78]{AM69}) applied to the Mayer-Vietoris exact sequence of Theorem \ref{thmMayerVietsequence}. 
\end{proof}

The alternating length identity can be checked completely in a small complete local example, reducing the eight module terms to explicit integers.

\begin{example}
Let $A=k\llbracket x,y\rrbracket$, $\mathcal M=(x,y)$, $I=(x)$, $J=(y)$, and $a=x+y$. Then $I\cap J=(xy)$ and
\[
A/((I\cap J)+aA)\cong k\llbracket x\rrbracket/(x^2),
\]
so the hypothesis of Proposition~\ref{propformula} holds. Multiplication by $a$ is injective modulo $(xy)$, $(x)$, and $(y)$, while it is the zero map on $A/(x,y)$. Hence the four kernel lengths are $0,0,0,1$. The four cokernel lengths are respectively $2,1,1,1$. Both alternating sums are therefore
\[
0-0-0+1=2-1-1+1=1.
\]
The formula reduces the compatibility of eight finite-length modules to this short numerical identity.
\end{example}

\section{Base change and norm loci}\label{sec:basechange}
In this section, we shift our focus to the geometric behavior of the regular denominator sets under base change. For this, we will need to consider an $A$-algebra $B$ that represents a family of spaces living over the base space $\Spec(A)$.

The following proposition describes the behavior of regular-denominator sets under base change along a flat ring homomorphism $f\colon A\to B$.

\begin{proposition}
    Let $f\colon A \rightarrow B$ be a flat ring homomorphism and let $I$ be an ideal of $A$. Then $f(S_I) \subseteq S_{f(I)B}$. Moreover, if $f$ is faithfully flat, then for every $a \in A$, we have that $a \in S_I$ if and only if $f(a) \in S_{f(I)B}$. 
\end{proposition}

\begin{proof}
Given $a\in S_I$, by definition of $S_I$, we have that the following exact sequence of $A$-modules:
  \begin{equation}\label{es1}
      0 \longrightarrow A/I \xrightarrow{\mu^{I}_a} A/I.
  \end{equation}
  Since $B$ is flat over $A$, tensoring with $B$ gives the exact sequence 
  \begin{equation}\label{es2}
  0 \longrightarrow (A/I) \otimes_A B \xrightarrow{\mu^I_a \otimes 1_B} (A/I) \otimes_A B.
  \end{equation}
  Under the canonical isomorphism $(A/I) \otimes_A B \cong B/f(I)B$, the map $\mu^{I}_a \otimes 1_B$ becomes the multiplication map $\mu^{f(I)B}_{f(a)}$ by $f(a)$ in $B$. Therefore, $\mu^{f(I)B}_{f(a)}$ is injective, which means $f(a)\in S_{f(I)B}$.

 Now assume that $f$ is faithfully flat. This means that the sequence (\ref{es1}) is exact if and only if the sequence (\ref{es2}) is exact. This means that $a \in S_I$ if and only if $f(a) \in S_{f(I)B}$.
\end{proof}

The difference between what flatness preserves and what faithful flatness can recover already appears in a very simple example over $\mathds{Z}$.

\begin{example}
For the faithfully flat extension $\mathds Z\to\mathds Z[t]$ and $I=(6)$, multiplication by $5$ is injective on both $\mathds Z/6\mathds Z$ and $(\mathds Z/6\mathds Z)[t]$, while multiplication by $2$ is noninjective on both. Thus regularity is preserved and reflected. In contrast, let $\mathds Z\to\mathds Z[1/2]$ be the flat, nonfaithful localization, take $I=(2)$, and take $a=2$. Multiplication by $2$ on $\mathds Z/2\mathds Z$ is not injective, but $I\mathds Z[1/2]=\mathds Z[1/2]$, so the base-changed quotient is zero and every multiplication map on it is injective. Flatness alone therefore cannot give the converse.
\end{example}

We now consider an ideal $I$ of the $A$-algebra $B$ such that the quotient $B/I$ is a finite locally free module over $A$. In this setting, we can guarantee that, for every $b\in B$, the set of prime ideals $P$ of the base ring $A$ for which the image of $b$ after base change to the fiber ring $B\otimes_A (A_P/PA_P)$ is a regular denominator with respect to the extension of $I$ is an open set in the Zariski topology. In particular, it is the basic open set associated to the determinant of multiplication by $b$.

\begin{proposition}\label{propdet}
Let $I\subseteq B$ be an ideal such that $B/I$ is a finite locally free module over $A$ and let $b\in B$. Consider the fiberwise regular-denominator locus:

$$U_{I}(b):=\{P\in \mathrm{Spec}(A)\mid b\otimes1\in S_{I(B\otimes_A (A_P/PA_P))}\},$$ 
where $b\otimes 1$ is the image of $b$ under base change to the fiber ring $B\otimes_A (A_P/PA_P)$.
Then $U_{I}(b)$ is open in $\Spec(A)$. In particular, 
$$U_{I}(b)=D(\operatorname{det}(\mu_b^I)).$$
\end{proposition}

\begin{proof}
    First, we observe that the quotient module $(B\otimes_A (A_P/PA_P))/I(B\otimes_A (A_P/PA_P))$ is canonically isomorphic to the module 
  $(B/I) \otimes_{A} (A_P/PA_P)$. This is because applying the right exact functor of tensor product  $-\otimes_{A} (A_P/PA_P)$ to the short exact sequence
  $$0 \longrightarrow I \longrightarrow B \longrightarrow B/I \longrightarrow 0$$
  yields the exact sequence
  $$I \otimes_{A} (A_P/PA_P) \longrightarrow B \otimes_{A} (A_P/PA_P) \longrightarrow (B/I) \otimes_{A}  (A_P/PA_P)  \longrightarrow 0.$$

Since $B/I$ is a finite locally free $A$-module, its base change $(B/I) \otimes_{A} (A_P/PA_P)$ to the field $A_P/PA_P$ is a finite-dimensional vector space over $A_P/PA_P$. Let us call this finite-dimensional vector space $V_{P}$. The condition $b\otimes 1 \in S_{I(B\otimes_A (A_P/PA_P))}$ then exactly means that the $A_P/PA_P$-linear map 
  $$\mu_{b}^I \otimes 1 \colon V_{P} \longrightarrow V_{P}$$is injective. But this happens if and only if its determinant is non-zero.
  Therefore, we have shown that $b\otimes 1 \in S_{I(B\otimes_A (A_P/PA_P))}$ if and only if  $\det(\mu_{b}^I \otimes 1) \neq 0$. Moreover, the determinant of $\mu_{b}^I \otimes 1$ is exactly the class of the global determinant $\det(\mu_{b}^I)$ inside the quotient $A_P/PA_P$. Hence, $\det(\mu_{b}^I \otimes 1) \neq 0$ precisely when $\det(\mu_{b}^I)$ is not in $P$. Thus, we get the desired conclusion by definition of the principal open sets of the Zariski topology.
\end{proof}

The next rank-two family replaces a separate regularity test in every fiber by the computation of one small multiplication determinant.

\begin{example}
Let $A=k[t]$, $B=A[x]$, $I=(x^2-t)$, and $b=x$. The quotient $B/I$ is free over $A$ with basis $1,x$, and multiplication by $x$ has matrix
\[
\begin{pmatrix}0&t\\1&0\end{pmatrix}.
\]
Its determinant is $-t$, so Proposition~\ref{propdet} gives
\[
U_I(x)=D(t)\subseteq\Spec(k[t]).
\]
At the fiber $t=0$, the class of $x$ is a nonzero nilpotent and multiplication is not injective; on every fiber with $t\ne0$, the determinant is nonzero. The exceptional locus is found from one scalar rather than by testing every fiber separately.
\end{example}

The following result shows that, in the presence of finite locally free modules over $A$, a multiplicative formula for determinants holds. Consequently, by Proposition \ref{propdet}, the intersection of the geometric loci where a certain element of the algebra is a regular denominator for two ideals $I$ and $J$ is equal to the intersection of the loci for $I\cap J$ and $I+J$. 

\begin{theorem}\label{detthm}
    Let $A$ be a ring, let $B$ be an $A$-algebra, and let $I,J\subseteq B$ be ideals such that $B/I,B/J,B/(I\cap J)$ and $B/(I+J)$ are finite locally free over $A$. Then for every $b\in B$, the following formula holds:
\begin{equation*}
\det(\mu_b^I)\det(\mu_b^J)=\det(\mu_b^{I\cap J})\det(\mu_b^{I+J}).
\end{equation*}
Consequently,
\begin{equation*}
\label{eq:family-open-equality}
U_I(b)\cap U_J(b)=U_{I\cap J}(b)\cap U_{I+J}(b).
\end{equation*}
\end{theorem}
\begin{proof}
    Consider the following well-known short exact sequence (see, for instance, 
    \cite[Chapter~II, \S1(7), Proposition~10]{Bou89})
    $$0\, \aar{} B/(I\cap J) \aar{\delta} {B/I \oplus B/J} \aar{\sigma} B/(I+J)\, \aar{} 0$$
    with $\delta(x+(I\cap J))=(x+I, x+J)$ and $\sigma (y+I,z+J)=(y-z)+(I+J)$. Since all the involved modules are finite locally free over $A$ by hypothesis, we apply the multiplicativity of determinants to the compatible multiplication endomorphisms of this short exact sequence. Since the determinant of the endomorphism on the middle term is the product of the determinants of the endomorphisms on the outer terms (see \cite[\S1]{KM76}) and (for the same reason) the determinant on a direct sum is the product of the determinants, the multiplicative formula for determinants is proved. The equality  
    $$U_I(b)\cap U_J(b)=U_{I\cap J}(b)\cap U_{I+J}(b)$$ 
    then follows immediately from Proposition \ref{propdet}.
\end{proof}

For a product algebra, the determinant identity is especially easy to see: everything reduces to a diagonal computation, and both the product formula and the corresponding equality of regular loci become immediate.

\begin{example}
Let $B=A\times A$, let $I=A\times0$, $J=0\times A$, and let $b=(r,s)$. Then
$
B/I\cong A$, $B/J\cong A$, $B/(I\cap J)\cong A^2$, and
$ B/(I+J)=0$.
The four determinants are $s$, $r$, $rs$, and $1$, respectively. The determinant formula of Theorem \ref{detthm} then becomes the identity $sr=(rs)\cdot1$, and the equality of loci becomes $D(s)\cap D(r)=D(rs)\cap\Spec(A)$. This simply says that a diagonal map is invertible exactly where both diagonal entries are invertible.
\end{example}

\section{Algorithmic implications}\label{sec:alg}
The results of the previous sections give several ways of simplifying the regularity test when the ideals involved have additional structure. In this section, with examples, we show how the previous results can be used to make calculations much shorter and easier.

Suppose, for instance, that $A$ is a polynomial ring over an effective field. To decide whether $a\in S_I$, one can test whether
$(I:a)=I$.
Thus, checking whether $a$ simultaneously belongs to $S_I$ and $S_J$ normally means carrying out two colon computations. The results of the previous sections show that this can sometimes be reduced considerably:
\begin{enumerate}
	\item if $a\in S_{I\cap J}$, it is enough to compare $(I:a)+(J:a)$ with $I+J$;
	\item if $I+J=A$, the two separate tests can be replaced by the single condition
	$(I\cap J:a)=I\cap J$
	using the equality $I\cap J=IJ$;
	\item in a finite locally free family, one can compute the determinant of multiplication and recover the entire regularity locus as a principal open subset of the base.
\end{enumerate}
There is a similar simplification when regular-denominator questions have to be answered repeatedly. Our Proposition~\ref{propdeneqantich} and Corollary~\ref{corollradical} allow the relevant denominator class to be recorded either by its finite antichain of maximal associated primes or by its canonical radical representative.

All of the operations appearing in these tests---ideal sums, intersections, quotients, and equality checks---are standard Gr\"obner-basis computations. By comparison, primary decomposition is a more elaborate procedure and may involve radicals, equidimensional decomposition, localization, factorization, and repeated Gr\"obner-basis calculations; see \cite{GTZ88,EHV92,GP08}. Since polynomial-ideal computations can have very poor worst-case complexity \cite{MM82}, we do not claim a general complexity improvement. The point is more modest and more practical: when the hypotheses of the preceding results are satisfied, one can often answer a specific regularity question without first computing the full associated-prime or primary-decomposition data.

Our first example shows how the comaximality case of Proposition~\ref{propchain} simplifies the computation: for two disjoint parallel components, two separate regularity tests can be replaced by a single calculation in the quotient by their product.

\begin{example}
Let $A=\mathds{R}[x,y]$, $I=(x)$, and $J=(x-1)$. These ideals define two parallel lines and satisfy $I+J=A$, so Proposition~\ref{propchain} gives
$	S_I\cap S_J=S_{I\cap J}$, where $I\cap J=IJ=(x^2-x)$. Let us
consider the polynomial
$
p=x^5y^7-x^4y^7+y=x^3y^7(x^2-x)+y.
$
By Chinese remainder theorem,  there is an isomorphism
\[
A/(x(x-1))\cong\mathds{R}[y]\times\mathds{R}[y]
\]
sending the class of $p$ to $(y,y)$. Since multiplication by $y$ is injective on each factor, $p\in S_{I\cap J}$, and hence $p\in S_I\cap S_J$. Note that the calculation uses one product quotient and avoids two separate regularity tests.
\end{example}

Interestingly enough, the same simplification from Proposition~\ref{propchain} still works when one of the ideals is defined by high-degree equations. In this case, the key observation comes from the much simpler linear generators, which make comaximality immediate.

\begin{example}
Let $A=\mathds{Q}[x,y,z]$. Let
$
I=(x^{10}+y^{10}-1,z-3)$ and $J=(z+3).$
Because $(z-3)-(z+3)=-6$ is a unit, $I+J=A$ and $I\cap J=IJ$. For the test element $s=x^{10}+y^{10}$, its image modulo $I$ is $1$, while its image modulo $J$ is the nonzero polynomial $x^{10}+y^{10}$ in the domain $\mathds{Q}[x,y]$. Hence $s\in S_I\cap S_J$, and Proposition~\ref{propchain} certifies the same fact by the single quotient
$
A/IJ$, where
\[
IJ=((x^{10}+y^{10}-1)(z+3),(z-3)(z+3)).
\]
Note that the benefit here is purely structural: once comaximality has been detected from the linear generators, one can work with a single product ideal.
\end{example}

When the ideals are not comaximal, the colon-sum criterion of Proposition~\ref{propeqfacts} can still bypass associated-prime and primary-decomposition computations, as the following monomial example demonstrates.

\begin{example}
Consider again the  ideals from Example \ref{monomialex} and let $a=u$. As computed earlier,
$
I\cap J=(xy)$ and $(I\cap J:u)=(xy)$,
so the hypothesis $u\in S_{I\cap J}$ is verified by one monomial colon. Moreover, we have
\[
(I:u)+(J:u)=(x,y)\ne(ux,uy,xy)=I+J.
\]
So the colon-sum criterion shows directly that $u$ is not a common regular denominator. Notice that the above computation only involves monomial ideals and an ideal-equality check, with no need to compute the associated primes of $A/I$ or $A/J$.
\end{example}

Our final example shows how the determinant identity of Theorem~\ref{detthm} can save a larger computation: the regularity locus for a rank-four quotient is recovered from two rank-two calculations, without ever having to write down the corresponding $4\times 4$ multiplication matrix.

\begin{example}
Let $R=\mathds{Q}[t]$, $B=R[x]$. Let
$
I=(x^2-t)$ and $J=(x^2-t+1).
$
Then $I$ and $J$ are comaximal with
$
I\cap J=IJ=((x^2-t)(x^2-t+1))$ and $B/(I+J)=0$.
For $b=x$,  with respect to the basis $1$, $x$ in each quotient, the multiplication maps on $B/I$ and $B/J$ have determinants
$
\det(\mu_x^I)=-t$ and
$\det(\mu_x^J)=1-t.
$
The determinant on the zero module is $1$. The multiplicative identity therefore gives the determinant on the rank-four quotient without constructing its $4\times4$ multiplication matrix:
\[
\det(\mu_x^{I\cap J})
=(-t)(1-t)=t^2-t.
\]
Consequently,
\[
U_I(x)\cap U_J(x)=D(t)\cap D(1-t)=D(t(t-1)).
\]
Thus, the two exceptional parameters $t=0$ and $t=1$ are obtained from two $2\times2$ determinants instead of a larger matrix over $R$.
\end{example}


\begin{thebibliography}{99}

\bibitem{AM69}
M.~F. Atiyah and I.~G. Macdonald,
\emph{Introduction to Commutative Algebra},
Addison--Wesley, Reading, MA, 1969.

\bibitem{Bou89}
N.~Bourbaki,
\emph{Algebra I: Chapters 1--3},
Elements of Mathematics,
Springer-Verlag, Berlin, 1989.

\bibitem{EHV92}
D. Eisenbud, C. Huneke, and W.~V. Vasconcelos,
Direct methods for primary decomposition,
\emph{Invent. Math.} \textbf{110} (1992), 207--235,
\href{https://doi.org/10.1007/BF01231331}{doi:10.1007/BF01231331}.

\bibitem{Eis95}
D. Eisenbud,
\emph{Commutative Algebra with a View Toward Algebraic Geometry},
Graduate Texts in Mathematics, vol.~150, Springer, New York, 1995.

\bibitem{Fer98}
D. Ferrand,
Un foncteur norme,
\emph{Bull. Soc. Math. France} \textbf{126} (1998), no.~1, 1--49,
\href{https://doi.org/10.24033/bsmf.2319}{doi:10.24033/bsmf.2319}.

\bibitem{GTZ88}
P. Gianni, B. Trager, and G. Zacharias,
Gr\"obner bases and primary decomposition of polynomial ideals,
\emph{J. Symbolic Comput.} \textbf{6} (1988), 149--167,
\href{https://doi.org/10.1016/S0747-7171(88)80040-3}{doi:10.1016/S0747-7171(88)80040-3}.

\bibitem{GP08}
G.-M. Greuel and G. Pfister,
\emph{A Singular Introduction to Commutative Algebra},
2nd ed., Springer, Berlin, 2008.

\bibitem{KM76}
F.~F. Knudsen and D. Mumford,
The projectivity of the moduli space of stable curves, I: Preliminaries on ``det'' and ``Div'',
\emph{Math. Scand.} \textbf{39} (1976), 19--55,
\href{https://doi.org/10.7146/math.scand.a-11642}{doi:10.7146/math.scand.a-11642}.

\bibitem{MM82}
E.~W. Mayr and A.~R. Meyer,
The complexity of the word problems for commutative semigroups and polynomial ideals,
\emph{Adv. Math.} \textbf{46} (1982), no.~3, 305--329,
\href{https://doi.org/10.1016/0001-8708(82)90048-2}{doi:10.1016/0001-8708(82)90048-2}.

\bibitem{Mat86}
H. Matsumura,
\emph{Commutative Ring Theory},
translated by M. Reid, Cambridge Studies in Advanced Mathematics, vol.~8,
Cambridge University Press, Cambridge, 1986.

\bibitem{Stk26}
The Stacks Project, (2026), \url{https://stacks.math.columbia.edu}.
\end{thebibliography}
\end{document}